\documentclass[12pt]{article}
 \usepackage{tikz}
 \usepackage{subcaption} 
 \usepackage{amsmath}
     \usepackage{bm}
     \allowdisplaybreaks[4]
     \usepackage[all]{xy}
     \usepackage{amssymb}
    \usepackage{amsthm}
    \usepackage{hyperref}
    \hypersetup{colorlinks=true,linkcolor=blue,citecolor=red}
    \usepackage{amscd}
    \usepackage{verbatim}
    \usepackage{eurosym}
    \usepackage{float}
    \usepackage{xcolor}
    \usepackage{dcolumn}
    \usepackage[mathscr]{eucal}
    \usepackage{pgfplots}
    \usetikzlibrary{calc}
    \usepackage{yhmath}
    \usepackage[T1]{fontenc}
    \usepackage[english]{babel}
    \makeatletter
    \DeclareTextSymbol{\cprime}{T1}{39} 
    \makeatother
    \usepackage{mathrsfs}
    \AtBeginDocument{%
      \providecommand{\bysame}{%
        \leavevmode\hbox to3em{\hrulefill}\thinspace
      }%
    }
    \usepackage{amsfonts,ifpdf}
    \usepackage{graphicx}
    \usepackage{times}
    \usepackage{epstopdf}
    \usepackage{cite}
    \usepackage{youngtab}
    \usepackage{ytableau}
    \usepackage{indentfirst}
    \usepackage{enumitem}
    \usepackage{caption}
    \usetikzlibrary{cd}
    \usepackage{soul}
    
    \ytableausetup
    {mathmode, boxsize=0.9em}
    
    \usepackage{listings}

\def\int{\mbox{\rm int}}

    \definecolor{codegray}{gray}{0.5}
    \definecolor{codegreen}{rgb}{0,0.6,0}
    \definecolor{codeblue}{rgb}{0.2,0.2,0.6}
    \definecolor{backcolour}{rgb}{0.95,0.95,0.95}
    
    \lstdefinestyle{mystyle}{
    	backgroundcolor=\color{backcolour},
    	commentstyle=\color{codegreen},
    	keywordstyle=\color{codeblue},
    	numberstyle=\tiny\color{codegray},
    	stringstyle=\color{codeblue},
    	basicstyle=\small\ttfamily,
    	breakatwhitespace=false,
    	breaklines=true,
    	captionpos=b,
    	keepspaces=true,
    	numbers=left,
    	numbersep=5pt,
    	showspaces=false,
    	showstringspaces=false,
    	showtabs=false,
    	tabsize=2
    }
    
    \newtheorem{theorem}{Theorem}[section]
    
    \newtheorem{lemma}[theorem]{Lemma}
    \newtheorem{prop}[theorem]{Proposition}
    \newtheorem{coro}[theorem]{Corollary}

   \theoremstyle{definition}
   \newtheorem{defi}[theorem]{Definition}

   \newtheorem{pf}{Proof}[section]
    \theoremstyle{remark}
   \newtheorem*{remark*}{Remark}

   \makeatletter \@addtoreset{equation}{section} \makeatother
   \makeindex 

\begin{document}
\begin{center}
{\Large\bf
 Convolution-type Identity for Characteristic Polynomials of  Geometric Semilattices}\\[7pt]
\end{center}
\vskip3mm

\begin{center}
Yanru Chen$^{1}$\quad Houshan Fu$^{2}$\quad Ang Li$^{3,*}$\quad Suijie Wang$^{4}$\\[8pt]
 
$^{1,3,4}$School of Mathematics, Hunan University\\
 Changsha 410082, Hunan, P. R. China\\[8pt]
 
$^{2}$School of Mathematics and Information Science, Guangzhou University\\
Guangzhou 510006, Guangdong, P. R. China\\[8pt]

$^4$Greater Bay Area Institute for Innovation, Hunan University\\
 Changsha 410082, Hunan, P. R. China\\[10pt]

$^*$Correspondence to be sent to: leonli@hnu.edu.cn  \\
$^{1}$yanruchen@hnu.edu.cn, $^{2}$fuhoushan@gzhu.edu.cn, $^{4}$wangsuijie@hnu.edu.cn\\[12pt]

\end{center}
\vskip 3mm 	
\begin{abstract}
We establish a convolution formula for the characteristic polynomial of a finite geometric semilattice $M$:
\[
\chi(M,st)=\sum_{X\in \underline{M}} s^{r-{\rm rk}_{\underline{M}}(X)}\chi(\underline{M}^X,t)\,\chi(M_{(X)},s),
\]
where $\underline{M}$ denotes the centralization of $M$, and $M_{(X)}$ denotes the localization at $X$. This generalizes a nice formula of Southerland, Southern, and Zhou, which is recovered at $s=1$. When specialized to hyperplane arrangements, the identity yields a new expansion closely related to Wang's convolution formula. We further provide a combinatorial interpretation of the convolution formula using the finite field method over $\mathbb{F}_{p^2}$ and $\mathbb{F}_p$.
\end{abstract}

\section{Introduction}
The characteristic polynomial $\chi(M,t)$ of a geometric semilattice $M$ was first defined in terms of its M\"{o}bius function by Rota \cite{rota1964}, and it encodes a wealth of fundamental combinatorial and geometric information about $M$.  Specifically, the {\em characteristic polynomial} $\chi(M,t)$ of a finite geometric semilattice $M$ of rank $r$ is defined as
\[
\chi(M,t):=\sum_{X\in M}\mu(\hat{0},X)t^{r-{\rm rk}_M(X)},
\]
where $\mu$ is the M\"{o}bius function of $M$, defined recursively by $\mu(X,X)=1$ and $\mu(X,Y)=\sum_{X\le Z<Y}-\mu(X,Z)$ for $X< Y$ in $M$, and  ${\rm rk}_M$ is the {\em rank function} of $M$. It is one of the most important and extensively studied polynomial invariants in graphs, geometric lattices (matroids), hyperplane arrangements, and geometric semilattices (semimatroids). In the setting of hyperplane arrangements, Zaslavsky's celebrated region-counting theorem in \cite{Zaslavsky1975Facing} expresses the number of all regions and the number of relatively bounded regions of a real hyperplane arrangement $\mathcal{A}$ as the evaluations $|\chi(\mathcal A,-1)|$ and $|\chi(\mathcal A,1)|$, respectively. Over finite fields, the characteristic polynomial also admits an explicit combinatorial interpretation via the finite field method independently developed by Athanasiadis \cite{athanasiadis1996}, and by Bj\"{o}rner and Ekedahl \cite{Bjorner-Ekedahl1997}. 

One central theme in this area is the study of convolution formulae for characteristic polynomials. In 1999,  Kook, Reiner and Stanton \cite{kook1999} established a convolution formula for the Tutte polynomial of matroids using incidence algebra methods, which was also independently discovered by Etienne and Las Vergnas \cite{etienne1998}. Subsequently, Kung presented a number of convolution-multiplication identities for multiplicative characteristic polynomials and Tutte polynomials of  graphs and matroids in  \cite{kung2004,kung2010}. More precisely, the multiplicative characteristic polynomial $\chi(M,st)$ of a matroid $M$ can be expressed as the sum, over all flats $X\in L(M)$, of the product of characteristic polynomials of its contraction $M^X$ and the submatroid  $M_X$ associated to $X$, i.e.,
\begin{equation}\label{GL-Convolution}
\chi(M,st)=\sum_{X\in L(M)} s^{r-{\rm rk}_M(X)}\chi(M^X,t)\chi(M_X,s),
\end{equation}
where $L(M)$ is the corresponding geometric lattice of $M$. In 2015, Wang \cite{wang2015} introduced the concept of M\"obius conjugation of posets as a unified method to reprove previous convolution formulae, and first provided a convolution formula for the multiplicative characteristic polynomial $\chi(\mathcal{A},st)$ of a hyperplane arrangement $\mathcal{A}$ in \cite[Theorem 2.1]{wang2015}:
\begin{equation}\label{HA-Convolution}
\chi(\mathcal{A},st) = \sum_{V\in L(\mathcal{A})}\chi(\mathcal{A}^V,t)\,\chi(\mathcal{A}_V,s),
\end{equation}
where $L(\mathcal{A})$ is the corresponding geometric semilattice (intersection poset) of $\mathcal{A}$, and $\mathcal{A}^V$ and $\mathcal{A}_V$ are its restriction and subarrangement associated with $V$ respectively, as presented in Definition \ref{HA-Cone}. Most recently, Fu extended this program to semimatroids in \cite{fu2025semimatroids}, and then placed it within the broader framework of multivariate Tutte polynomials in \cite{fu2025multivariate}.

In \cite{ehrenborg2019counting}, Ehrenborg first introduced the concept of level for faces of real hyperplane arrangements $\mathcal{A}$, and then enumerated the faces of extended Shi arrangements by both dimension and level. Subsequently, Zaslavsky \cite{zaslavsky2003ideal} gave a counting formula for the number of faces of dimension $n$ and level $l$ in terms of the characteristic and Whitney polynomials of some related arrangements. This result refines his classical region-counting formula, and is now commonly referred to as the Zaslavsky level-counting formula. Most recently, Southerland, Southern, and Zhou \cite{southerland2025region} restated the Zaslavsky level-counting formula using a construction called the centralization $\underline{\mathcal{A}}$ of $\mathcal{A}$, provided a bijective proof, and further applied it to give a nice expression for the characteristic polynomial $\chi(\mathcal{A},t)$ in \cite[Theorem 6.4]{southerland2025region}:
\begin{equation}\label{CP-Convolution}
\chi(\mathcal{A},t) = \sum_{V \in L(\underline{\mathcal{A}})}\chi(\underline{\mathcal{A}}^V,t)\chi(\mathcal{A}_{(V)},1),
\end{equation}
where $\mathcal{A}_{(V)}$ is the localization of $\mathcal{A}$ to $V$, as given in Definition \ref{HA-Cone}. Additionally, they extended the two constructions of centralization and cone of a hyperplane arrangement to the broader framework of geometric semilattices $M$, and obtained an analogous formula for the characteristic polynomial $\chi(M,t)$ in \cite[Theorem 1.4]{southerland2025region}:
\begin{equation}\label{GSL-Convolution}
\chi(M,t)=\sum_{X\in\underline{M}}\chi({\underline{M}^X},t)\chi(M_{(X)},1),
\end{equation}
where $\underline M$  and $M_{(X)}$ are the centralization and  the localization of $M$ at $X$, respectively, as presented in Definition \ref{Def-Cone}.

Motivated by Southerland, Southern, and Zhou's work, our main result is a convolution formula for the multiplicative characteristic polynomial $\chi(M,st)$ of a finite geometric semilattice $M$, based on the perspectives of centralization and localization of geometric semilattices. Comparing this with the standard convolution formula \eqref{GL-Convolution}, our formula ranges over its centralization  
$\underline M$, rather than over $M$ itself, and also agrees with  \eqref{GL-Convolution} when $M$ is a geometric lattice (matroid).
\begin{theorem}\label{main thm}
Let $M$ be a finite geometric semilattice of rank $r$. Then
\[
\chi(M,st)=\sum_{X\in \underline{M}} s^{\,r-{\rm rk}_{\underline{M}}(X)}\chi(\underline{M}^X,t)\,\chi(M_{(X)},s).
\]
\end{theorem}

It is worth noting that by taking $s=1$ in Theorem \ref{main thm}, we can recover Southerland, Southern, and Zhou's formula \eqref{GSL-Convolution}. Specialized to hyperplane arrangements, Theorem \ref{main thm} directly yields a new convolution formula for the multiplicative characteristic polynomial $\chi(\mathcal{A},st)$, presented in Corollary \ref{cor:arrangement}, which closely resembles Wang's original formula \eqref{HA-Convolution}. Further setting $s=1$ in Corollary \ref{cor:arrangement} recovers Southerland, Southern, and Zhou's convolution formula for the characteristic polynomial $\chi(\mathcal{A},t)$ in \eqref{CP-Convolution}, and also shows that \eqref{CP-Convolution} holds over any field $\mathbb{F}$, not just $\mathbb{R}$.
\begin{coro}\label{cor:arrangement}
Let $\mathcal{A}$ be a hyperplane arrangement in $\mathbb{F}^n$.  Then
\[
\chi(\mathcal{A},st) = \sum_{V \in L(\underline{\mathcal{A}})}\chi(\underline{\mathcal{A}}^V,t)\,\chi(\mathcal{A}_{(V)},s).
\]
\end{coro}
A natural question is to seek combinatorial, geometric, or algebraic interpretations of these convolution formulae. Such interpretations appear extremely challenging for abstract geometric semilattices. When restricting our attention to hyperplane arrangements,  Section \ref{Sec4} gives the combinatorial interpretations for the convolution formulae of multiplicative characteristic polynomials presented in \eqref{HA-Convolution} and Corollary \ref{cor:arrangement} via the finite field method \cite{athanasiadis1996,Bjorner-Ekedahl1997}. 
   
The paper is organized as follows. Section \ref{Sec2} collects the necessary definitions on geometric semilattices and hyperplane arrangements, and then establishes two structure theorems concerning an arbitrary geometric semilattice and its cone. Section \ref{Sec3} is devoted to proving Theorem \ref{main thm}. Section \ref{Sec4} focuses on presenting the combinatorial interpretations of these convolution formulae.
\section{Structure theorems}\label{Sec2}
In this section, we focus on the necessary definitions of geometric semilattices and hyperplane arrangements, and then establish two structure theorems for any geometric semilattice and its associated cone.
\subsection{Preliminaries}\label{Sec2-1}
We briefly recall some necessary definitions regarding geometric semilattices and hyperplane arrangements. For further details, we refer the reader to \cite{stanley2007hyperplane,wachs1986}. 
\begin{defi}
A ranked lattice $L$ is called \emph{semimodular} if for all $s,t\in L$,
\[
\mathrm{rk}_L(s\wedge t)+\mathrm{rk}_L(s\vee t)\le \mathrm{rk}_L(s)+\mathrm{rk}_L(t).
\]
If every element $s\in L$ can be written as a join of some atoms, then $L$ is called \emph{atomistic}. $L$ is called a \emph{geometric lattice}
if it is both semimodular and atomistic.
\end{defi}

\begin{defi}
Let $M$ be a ranked meet-semilattice. A set $S$ of atoms of $M$ is \emph{independent} if they have an upper bound $\bigvee S$ in $M$ and ${\rm rk}_M\big(\bigvee S\big)=|S|$. $M$ is called a \emph{geometric semilattice} if it satisfies
\begin{itemize}
\item every principal order ideal is a geometric lattice;
\item whenever $S$ is an independent set of atoms of $M$ and $t\in M$ satisfies ${\rm rk}_M(t)<{\rm rk}_M\big(\bigvee S\big)$, there exists some $a\in S$ such that $a\nleq t$ and $t\vee a$ exists in $M$.
\end{itemize}
\end{defi}
   
\begin{defi}
Let $M$ be a finite geometric semilattice. For each $X\in M$, the {\em principal order ideal } $M_X$ and the {\em principal dual order ideal } $M^X$ are respectively
\[
M_X:=\{Y\in M: Y\le X\}\quad{\rm and}\quad M^X:=\{Y\in M:Y\ge X\}.
\]
\end{defi}
A {\em hyperplane arrangement} $\mathcal{A}$ is a finite collection of affine hyperplanes in an $n$-dimensional vector space $\mathbb{F}^n$. The {\em intersection poset} $L(\mathcal{A})$ of $\mathcal{A}$ is a poset consisting of all nonempty intersections of  some hyperplanes from $\mathcal{A}$, ordered by reverse inclusion. In fact, the intersection posets of hyperplane arrangements are examples of geometric semilattices, and intersection posets of central arrangements are geometric lattices. Note that geometric lattices and geometric semilattices correspond to simple matroids and semimatroids in \cite{ardila2007} respectively, though we do not use that language in this paper. 

Most recently, using the centralization of hyperplane arrangements, Southerland, Southern and Zhou \cite{southerland2025region} revisited Zaslavsky's level-counting formula.
\begin{defi}\label{HA-Cone}
For a hyperplane \(H = \{\bm x \in \mathbb{F}^n : \bm a\cdot \bm x = b \}\) in $\mathbb{F}^n$, the \emph{centralization} of \(H\) is 
\[
\underline{H} := \{ \bm x \in \mathbb{F}^n : \bm a \cdot \bm x = 0\},
\]
which is the unique parallel hyperplane passing through the origin. Let $\mathcal{A}= \{H_1, \dots, H_m\}$ be a hyperplane arrangement in $\mathbb{F}^n$. The \emph{centralization} of $\mathcal{A}$ is given by 
\[
\underline{\mathcal{A}} := \{\underline{H_1}, \dots, \underline{H_m}\}.
\]
For a linear subspace $V \subseteq \mathbb{F}^n$, the \emph{localization} and \emph{restriction} of $\mathcal{A}$ at $V$ are respectively
\[
\mathcal{A}_{(V)} := \{H \in \mathcal{A} : V \subseteq H \text{ or } V \cap H = \emptyset\}\quad{\rm and}\quad 
\mathcal{A}^V := \{H \cap V : H \in \mathcal{A}\} - \{\emptyset, V\}.
\]
Naturally, each flat $V\in L(\mathcal{A})$ induces a subarrangement $\mathcal{A}_V$ of  $\mathcal{A}$ defined by
\[
\mathcal{A}_V:= \{H \in \mathcal{A} : V \subseteq H\}.
\]
It is obvious that $\mathcal{A}_V$ is also a subarrangement of the localization $\mathcal{A}_{(V)}$.
\end{defi}
Furthermore, Southerland, Southern, and Zhou \cite{southerland2025region} introduced the cone $cM$ and the centralization $\underline{M}$ of an arbitrary geometric semilattice $M$, as the abstract structures corresponding to the cone and centralization of an affine arrangement, respectively. They also showed that the cone $cM$ and the centralization $\underline{M}$ are both geometric lattices.
\begin{defi}\label{Def-Cone}
Let $M$ be a geometric semilattice with atoms $A=\{a_1,\dots,a_k\}$. Associated with a distinguished additional `atom' $a_0$, for each \(X\in M\), define $X^*:=\{a\in A: a \text{ and } X \text{ have no common upper bound}\}$, and $\underline{X}:=X\sqcup X^*\sqcup\{a_0\}$. The \emph{cone} $cM$ of \(M\) is the poset
   		\[
   		cM:=\{X:\ X\in M\}\sqcup\{\underline{X}:\ X\in M\},
   		\]
   		ordered by inclusion. The \emph{centralization} $\underline{M}$ of \(M\) is the subposet of $cM$ defined as
   		\[
   		\underline{M}:=\{\underline{X}: X\in M\}.
   		\]
In addition, for any $X\in \underline{M}$, the {\em localization} $M_{(X)}$ of $M$ at $X$ is given by
\[
M_{(X)}:=(cM)_X-\underline{M}.
\]
\end{defi}	
 \subsection{Structure Theorems}\label{Sec2-2}
Let us start with a basic exchange property concerning atoms of a geometric lattice, which will be needed later.
\begin{lemma}\label{lem:atomic-exchange}
For a geometric semilattice $M$, let \(cM\) be as in Definition \ref{Def-Cone}.  For any $X\in cM$, and atoms $a,b$ of $cM$, if $a\nleq X$ and $b\nleq X$, then
\[
a\le X\vee_{cM} b \quad\Longleftrightarrow\quad b\le X\vee_{cM} a.
\]
\end{lemma}
\begin{proof}
Suppose $b\le X\vee_{cM} a$. Since $b\nleq X$, we have $X<X\vee_{cM} b\le X\vee_{cM} a$. Note that $X\vee_{cM} a$ covers $X$ by \cite[Lemma 5.6]{southerland2025region}. Therefore, there is no element strictly between $X$ and $X\vee_{cM} a$. This implies that $X\vee_{cM} a=X\vee_{cM} b$, and hence $a\le X\vee_{cM} b$. By symmetry, the converse holds as well.
\end{proof}

Next, we give an explicit characterization of the elements in the centralization of a geometric semilattice.
\begin{prop}\label{prop:bar-is-join}
For a geometric semilattice $M$, let \(cM\) be as in Definition \ref{Def-Cone}. For each \(X\in M\), we have $\underline{X}=X\vee_{cM} a_0$.
\end{prop}
\begin{proof}
Since $\underline{X}=X\sqcup X^*\sqcup\{a_0\}$, we immediately have $X\vee_{cM}a_0\le\underline{X}$. It remains to prove that  $\underline{X}\le X\vee_{cM}a_0$. If $X^*=\emptyset$, then we directly have $\underline{X}\le X\vee_{cM}a_0$. If $X^*\ne\emptyset$, for any $a\in X^*$, then $a$ and $X$ have no common upper bound in $M$, and hence $X\vee_{cM} a\in\underline{M}$. It follows that \(a_0\le X\vee_{cM} a\). Applying Lemma \ref{lem:atomic-exchange}, we deduce $a\le X\vee_{cM} a_0$. Consequently, we have $X\vee_{cM}a\le X\vee_{cM}a_0$. This implies that $\underline{X}\le X\vee_{cM}a_0$. 
\end{proof}
   	
Now, we proceed to discuss several basic properties of a geometric semilattice and its cone.
\begin{lemma}\label{Three-Properties}
For a geometric semilattice $M$, let \(cM\) be as in Definition \ref{Def-Cone}. The following properties hold.
\begin{itemize}
\item[{\rm (a)}] If $X,Y\in M$ satisfy $Y\le\underline{X}\le \underline{Y}$ in $cM$, then $\underline{X}=\underline{Y}$.
\item[{\rm (b)}] If \(X,Y\in M\) satisfy $Y\le X\le\underline{Y}$ in $cM$, then $X=Y$.
\item[{\rm (c)}] If \(X,Y\in M\) satisfy $X\le \underline Y$ in $cM$, then there exists \(Z\in M\) such that $X\le Z$ and $\underline Z=\underline Y$.
\end{itemize}
\end{lemma}
\begin{proof}
From Proposition \ref{prop:bar-is-join}, we have $\underline{Y}=Y\vee_{cM} a_0$ and $\underline{X}=X\vee_{cM} a_0$. Since $X,Y\in M$ satisfy $Y\le \underline{X}\le \underline{Y}$, we have that $Y\ne \underline{X}$, and $\underline{Y}$ covers $Y$ in $cM$ via \cite[Lemma 5.6]{southerland2025region}. It follows that $\underline{X}=\underline{Y}$. Similar to the proof of part (a), one can show that part (b) holds. It remains to prove part (c). Choose a subset $Z\subseteq X\cup Y$ to be the maximal element of  $M$ containing $X$. Consequently, we have $Z\le \underline Y$ in $cM$. It follows from Proposition \ref{prop:bar-is-join} that $\underline Z\le\underline Y$ in $cM$. Conversely, taking an element $a\in Y$, if $a\in Z$, then $a\in\underline Z$. If $a\notin Z$, then $Z$ and $a$ have no common upper bound in $M$ from the maximality of $Z$. Thus, $a\in Z^*$ and hence $a\in\underline Z$. This implies that $\underline Y\le\underline Z$. Consequently, $\underline Z=\underline Y$. 
\end{proof}

With the preparations above, we now have all the tools in place to prove our two main structure theorems, which establish the relationship between certain subposets of a geometric semilattice $M$ and the corresponding subposets of its cone $cM$. These also serve as key ingredients in the proof of Theorem \ref{main thm}.
\begin{theorem}\label{isomo_1}
For a geometric semilattice $M$, let \(cM\) be as in Definition \ref{Def-Cone}, and  \(X\in \underline M\). Then
\[
(cM)^X=\underline M^X\quad{\rm and}\quad(cM)_X=c(M_{(X)}).
\]
\end{theorem}
\begin{proof}
We first show that $(cM)^X=\underline M^X$. Since \(X\in \underline M\), we have 
\[
(cM)^X=\{Y\in cM:X\le Y\}=\{Y\in \underline M:X\le Y\}=\underline M^X.
\]

Next, we prove that $(cM)_X=c(M_{(X)})$. Note that $M_{(X)}=\{Y\in M:Y\le X\}$ and $(cM)_X$ can be decomposed into the following two parts:
\[
(cM)_X=M_{(X)}\sqcup\{Y\in\underline M:Y\le X\}.
\]
On the other hand, the set $A_X$ of atoms of $M_{(X)}$ is $A_X:=X-\{a_0\}$. From Definition \ref{Def-Cone}, the cone $c(M_{(X)})$ of $M_{(X)}$ is given by
\[
c(M_{(X)})=M_{(X)}\sqcup \{\underline{Y_X}:Y_X\in M_{(X)}\},
\]
where $Y_X^*=\{a\in A_X:Y_X \text{ and } a \text{ have no common upper bound in } M_{(X)}\}$ and $\underline{Y_X}=Y_X\sqcup Y_X^*\sqcup\{a_0\}$. By comparing $(cM)_X$ with $c(M_{(X)})$, it remains to verify
that
\[
\{Y\in\underline M:Y\le X\}= \{\underline{Y_X}:Y_X\in M_{(X)}\}.
\]

If $Y\le X$ with $Y\in\underline{M}$, then there exists $Z\in M$ such that $Y=\underline{Z}$, $Z^*\subseteq A_X$, and $Z\le X$ in $cM$. Recall that $Z^*$ is defined as the set of atoms $a$ of $M$ such that $Z$ and $a$ have no common upper bound in $M$. Consequently, for any $a\in Z^*$, $Z$ and $a$ have no common upper bound in $M_{(X)}$ as well. On the other hand, for any $a\in A_X-(Z\sqcup Z^*)$, $Z$ and $a$ have a common upper bound in $M$, and $Z\vee_M a\le X$ holds in $cM$. This means that $Z\vee_M a\in M_{(X)}$ for any $a\in A_X-(Z\sqcup Z^*)$, which is therefore a common upper bound of $Z$ and $a$ in $M_{(X)}$. Summarizing the above arguments, we conclude that $Z^*$ is exactly the set of atoms $a$ of $M_{(X)}$ such that $Z$ and $a$ have no common upper bound in $M_{(X)}$. Thus, $Y=Z\sqcup Z^*\sqcup\{a_0\}\in \{\underline{Y_X}:Y_X\in M_{(X)}\}$.

Conversely, for any $Y=\underline{Y_X}$ with $Y_X\in M_{(X)}$, since $Y_X\le X$, according to  part (c) in Lemma \ref{Three-Properties}, there exists $Z\in M$ for which $Y_X\le Z$ in $M$ and  $X=\underline{Z}$. On one hand, for any $a\in A_X-(Y_X\sqcup Y_X^*)$, since $Y_X$ and $a$ have a common upper bound in $M_{(X)}$,  it follows that $Y_X$ and $a$ naturally have a common upper bound in $M$. On the other hand, for any $a\in A-A_X$, as $Z\vee_M a$ always exists, it is a common upper bound of  $Y_X$ and $a$ in $M$. Thus, $Y_X\vee_M a$ always exists for all $a\in A-(Y_X\sqcup Y_X^*)$. Finally, we claim that $Y_X$ and $a$ have no common upper bound in $M$ for any $a\in Y_X^*$. Otherwise, $Y_X\vee_M a$ exists in $M$ and $Y_X\vee_M a\le X$. This implies that $Y_X\vee_M a\in M_{(X)}$, contradicting the definition of $Y_X^*$. Collecting the previous arguments, we conclude that $Y_X^*$ is exactly the set of atoms $a$ of $M$ such that $Y_X$ and $a$ have no common upper bound in $M$. Consequently, $Y\in \{Z\in\underline M:Z\le X\}$. This completes the proof.
\end{proof}

\begin{theorem}\label{isomo_2}
For a geometric semilattice $M$, let \(cM\) be as in Definition \ref{Def-Cone}, and \(X\in M\). Then
\[
(cM)_X=M_X\quad{\rm and}\quad(cM)^X=c(M^X).
\]
\end{theorem}
\begin{proof}
Since $X\in M$, we directly derive
\[
(cM)_X=\{Y\in M:Y\le X\}=M_X.
\]
Therefore, the first part holds. 

Next, we verify that $(cM)^X=c(M^X)$. As $X\in M$,  $(cM)^X$ can be decomposed as:
\[
(cM)^X=M^X\sqcup\{Y\in\underline M:X\le Y\}.
\]
Together with part (c) of Lemma \ref{Three-Properties}, we rewrite $(cM)^X$ as: 
\[
(cM)^X=M^X\sqcup\{\underline Y: Y\in M^X\}.
\]
On the other hand, let $A^X:=A-(X\sqcup X^*)$. Then $X\vee_M a$ covers $X$ in $M$ for each $a\in A^X$, and hence $\{X\vee_M a:a\in A^X\}$ is exactly the set of all atoms of $M^X$.
For every $Y^X\in M^X$, $(Y^X)^*$ is given by
\[
(Y^X)^*=\{a\in A^X: Y^X \text{ and } (X\vee_M a) \text{ have no common upper bound in } M^X\}\sqcup X.
\] 
We now set $\underline{Y^X}=Y^X\sqcup (Y^X)^*\sqcup\{A_0\}$ for every $Y^X\in M^X$, where $A_0=X\vee_{cM} a_0$ denotes the distinguished additional atom in the construction of the cone \(c(M^X)\). Indeed, from parts (a) and (b) of Lemma \ref{Three-Properties},  $A_0=X\vee_{cM} a_0=\underline X$ is the unique element of $cM-M$ covering $X$. Therefore $A_0 =\underline{X}= X \sqcup X^* \sqcup \{a_0\}$. Notice from $X\le Y^X$ in $M$ that for any $a\in A^X$, $Y^X$ and $(X\vee_M a)$ have no common upper bound in  $M^X$ if and only if $Y^X$ and $a$ have no common upper bound in $M$. Thus, $\underline{Y^X}$ can be expressed as 
\begin{equation}\label{A=B}
\underline{Y^X}=Y^X\sqcup \{a\in A^X: Y^X\vee_M a \text{ does not exist in }  M\}\sqcup X^*\sqcup\{a_0\},
\end{equation}
and
\[
c(M^X)=M^X\sqcup\{\underline{Y^X}:Y^X\in M^X\}.
\]	
Comparing $(cM)^X$ and $c(M^X)$, it is enough to show that
\[
\{\underline Y: Y\in M^X\}=\{\underline{Y^X}:Y^X\in M^X\}.
\]

First consider $\underline Y$ with $Y\in M^X$. Since $X\le Y$ in $M$, it is straightforward to check that $X^*\subseteq Y^*$, $Y^*-X^*\subseteq A^X$ and
\[
A^X=A^Y\sqcup (Y-X)\sqcup(Y^*-X^*).
\]
It follows that $Y^*-X^*$ is exactly the set of atoms $a$ in $A^X$ such that $Y$ and $a$ have no common upper bound in $M$. Together with \eqref{A=B}, we obtain that $\underline Y\in \{\underline{Y^X}:Y^X\in M^X\}$.

Conversely, take $\underline{Y^X}$ with $Y^X\in M^X$. Then $X\le Y^X$ in $M$, and it suffices to prove that $\underline{Y^X}-Y^X-\{a_0\}$ is the set of atoms $a$ of $M$ for which $Y^X$ and $a$ have no common upper bound in $M$. It follows from \eqref{A=B} that for any $a\in \underline{Y^X}-Y^X-\{a_0\}$, $Y^X\vee_M a$ does not exist in $M$. On the other hand, we have $A-\underline{Y^X}\subseteq A^X$ since $X\sqcup X^*\subseteq \underline{Y^X}$. Thus, for any $a\in A-\underline{Y^X}$, $Y^X$ and $X\vee_M a$ have a common upper bound in $M^X$, which implies that $Y^X$ and $a$ also have a common upper bound in $M$. Therefore, $\underline{Y^X}-Y^X-\{a_0\}$ is indeed the set of atoms $a$ of $M$ such that $Y^X$ and $a$ have no common upper bound in $M$. Consequently, $\underline{Y^X}\in \{\underline Y: Y\in M^X\}$. This completes the proof.
\end{proof}

\section{Proof of Theorem \ref{main thm}}\label{Sec3}
In this section, we aim to verify Theorem \ref{main thm}. To this end, we first derive the counterparts of the convolution formulae for multiplicative characteristic polynomials of matroids in \cite{kung2004} and semimatroids in \cite{fu2025semimatroids} in the setting of geometric semilattices.
\begin{prop}\label{convolution over semimatroid}
Let $M$ be a finite geometric semilattice of rank $r$. Then
\[
\chi(M,st)=\sum_{X\in M}s^{r-{\rm rk}_M(X)}\chi({M^X},t)\chi({M_X},s).
\]
\end{prop}
\begin{proof}
Note that \(\chi(M^X,t)\) and \(\chi(M_X,s)\) can be expressed respectively as
\[
\chi(M^X,t)=\sum_{X\le Y}\mu(X,Y)t^{r-{\rm rk}_M(Y)}\;{\rm and}\; \chi(M_X,s)=\sum_{Z\le X}\mu(\hat0,Z)s^{{\rm rk}_M(X)-{\rm rk}_M(Z)}.
\]
It follows that the right-hand side can be written as
\begin{align*}
&\sum_{X\in M}s^{r-{\rm rk}_M(X)}\Big(\sum_{X\le Y}\mu(X,Y)t^{r-{\rm rk}_M(Y)}\Big)
\Big(\sum_{Z\le X}\mu(\hat0,Z)s^{{\rm rk}_M(X)-{\rm rk}_M(Z)}\Big)\\
&=\sum_{X\in M}\sum_{Z\le X\le Y}\mu(\hat0,Z)\mu(X,Y)s^{r-{\rm rk}_M(Z)}t^{r-{\rm rk}_M(Y)}\\
&=\sum_{Z\in M}\mu(\hat0,Z)s^{r-{\rm rk}_M(Z)}\sum_{Z\le X\le Y}\mu(X,Y)t^{r-{\rm rk}_M(Y)}.
\end{align*}
According to the definition of the M\"obius function, the sum $\sum_{Z\le X\le Y}\mu(X,Y)$ equals $0$ if $Z\ne Y$. Consequently, the aforementioned sum is simplified as
\[
\sum_{Z\in M}\mu(\hat0,Z)(st)^{r-{\rm rk}_M(Z)}=\chi(M,st),
\]
which completes the proof.
\end{proof}
   	
The following Crosscut Theorem provides a specific way to compute the M\"obius function of a lattice, which enables us to derive an alternative expression of the characteristic polynomial of a geometric semilattice.
\begin{prop}[Crosscut Theorem, {\cite[Theorem 3]{rota1964}}]\label{thm:crosscut}
Let $L$ be a finite lattice with minimum element $\hat 0$, maximum element $\hat 1$, and atom set $A$. Then
\[
\mu(\hat 0,\hat 1)=\sum_{B\subseteq A,\,\bigvee B=\hat 1}(-1)^{|B|}.
\]
\end{prop}

Let $M$ be a geometric semilattice with atom set $A$. It is worth noting that every principal order ideal $M_X$ for $X\in M$ forms a geometric lattice. By applying the Crosscut Theorem, we immediately obtain
\[
\mu(\hat 0,X)
=
\sum_{B\subseteq A,\,\bigvee B=X}(-1)^{|B|}.
\]
Substituting this into the definition of the characteristic polynomial directly yields the following formula. This allows us to establish the connection between characteristic polynomials of a geometric semilattice and its cone.
\begin{prop}\label{semilattice}
Let $M$ be a geometric semilattice of rank $r$ with atom set $A$. Then the characteristic polynomial $\chi(M,t)$ can be expressed as
\[
\chi(M,t)=\sum_{B\subseteq A,\,\bigvee B\text{ exists }}(-1)^{|B|}t^{r-{\rm rk}_M(\bigvee B)}.
\]
\end{prop}

A well known result shows that the characteristic polynomials of an affine arrangement $\mathcal{A}$ and its cone $c\mathcal{A}$ satisfy the relation:
\[
\chi(c\mathcal{A},t)=(t-1)\chi(\mathcal{A},t),
\]
see \cite{stanley2007hyperplane}. In 2007, Ardila generalized this identity to characteristic polynomials of a semimatroid and its associated pointed matroid in \cite[Proposition 8.7]{ardila2007}. In what follows, we establish its counterpart for geometric semilattices, which will be applied in subsequent arguments.

\begin{prop}\label{cM-M}
Let $M$ be a geometric semilattice of rank $r$, and let \(cM\) be as in Definition~\ref{Def-Cone}. Then
\[
\chi(cM,t)=(t-1)\chi(M,t).
\]
\end{prop}
   
\begin{proof}
By the construction of \(cM\), it is a geometric lattice of rank \(r+1\). From Proposition \ref{semilattice}, the characteristic polynomial $\chi(cM,t)$ can be written as
\[
\chi(cM,t)=\sum_{B\subseteq A\cup\{a_0\}}(-1)^{|B|}t^{r+1-\operatorname{rk}_{cM}(\bigvee_{cM}B)}.
\]
We split the sum depending on whether the subsets contain $a_0$, and then obtain
\[
\chi(cM,t)=\sum_{B\subseteq A}(-1)^{|B|}\big(t^{r+1-\operatorname{rk}_{cM}(\bigvee_{cM}B)}-
t^{r+1-\operatorname{rk}_{cM}(\bigvee_{cM}(B\cup\{a_0\}))}\big).
\]
Given a subset $B\subseteq A$, if $\bigvee_MB$ exists, set $X=\bigvee_MB$. Then, we have $X=\bigvee_{cM}B$ and $\bigvee_{cM}(B\cup\{a_0\})=X\vee_{cM}a_0=\underline X$ from Proposition \ref{prop:bar-is-join}. Since \(a_0\nleq X\), it follows that
\[
\operatorname{rk}_{cM}(\underline X)=\operatorname{rk}_{cM}(X\vee_{cM} a_0)=\operatorname{rk}_{cM}(X)+1=\operatorname{rk}_M (X)+1.
\]
Thus, the expression  $t^{r+1-\operatorname{rk}_{cM}(\bigvee_{cM}B)}-t^{r+1-\operatorname{rk}_{cM}(\bigvee_{cM}(B\cup\{a_0\}))}$ can be simplified to $(t-1)t^{r-\operatorname{rk}_M (X)}$. When $\bigvee_MB$ does not exist, then \(\bigvee_{cM}B\notin M\), and hence \(a_0\le \bigvee_{cM}B\). This implies  $\bigvee_{cM}(B\cup\{a_0\})=\bigvee_{cM}B$. Consequently, $t^{r+1-\operatorname{rk}_{cM}(\bigvee_{cM}B)}-t^{r+1-\operatorname{rk}_{cM}(\bigvee_{cM}(B\cup\{a_0\}))}$ equals $0$ in this case. Summarizing the above arguments, we conclude that 
\[
\chi(cM,t)=(t-1)\sum_{B\subseteq A,\,\bigvee_M B\text{ exists}}(-1)^{|B|}t^{r-\operatorname{rk}_M(\bigvee_M B)}=(t-1)\chi(M,t),
\]
where the last equality follows from Proposition \ref{semilattice}. This completes the proof.
\end{proof}

With the above preparations, we now have enough tools to prove our main result as follows:
\[\chi(M,st)=\sum_{X\in \underline{M}} s^{\,r-{\rm rk}_{\underline{M}}(X)}\chi(\underline{M}^X,t)\,\chi(M_{(X)},s).\]
\begin{proof}[Proof of Theorem \ref{main thm}]
From Propositions \ref{convolution over semimatroid} and \ref{cM-M}, we immediately obtain
\[
(st-1)\chi(M,st)=\sum_{X\in cM}s^{r+1-\operatorname{rk}_{cM}(X)}\chi\big((cM)^X,t\big)\chi\big((cM)_X,s\big).
\]
We can split the sum into two parts according to $cM=M\sqcup\underline M$:
\begin{align}
(st - 1)\chi(M,st)&= \sum_{X \in \underline M}s^{r+1-{\rm rk}_{cM}(X)}\chi\big((cM)^X, t\big)\chi\big((cM)_X, s\big)\label{A=B1}\\
& + \sum_{X \in M}s^{r+1-{\rm rk}_{cM}(X)}\chi\big((cM)^X, t\big)\chi\big((cM)_X, s\big)\label{A=B2}. 
\end{align}

First consider the sum in \eqref{A=B1}. Since $X\in\underline M$, it follows from  Theorem \ref{isomo_1} that
\[(cM)^X= \underline{M}^X\quad{\rm and}\quad(cM)_X= c(M_{(X)}).\]
Thus, we have $\chi\big((cM)^X, t\big)=\chi(\underline{M}^X, t)$ and $\chi\big((cM)_X, s\big)=\chi\big(c(M_{(X)}),s\big)$. Together with Proposition \ref{cM-M}, we arrive at $\chi\big((cM)_X, s\big)=(s-1)\chi(M_{(X)},s)$. Substituting these into the sum in \eqref{A=B1} and using \({\rm rk}_{cM}(X)={\rm rk}_{\underline{M}}(X)+1\) with \(X\in\underline{M}\), we deduce
\[
\sum_{X\in\underline M}s^{r+1-{\rm rk}_{cM}(X)}\chi\big((cM)^X, t\big)\chi\big((cM)_X, s\big)
= (s-1)\sum_{X \in \underline{M}} s^{r-{\rm rk}_{\underline{M}}(X)}\chi(\underline{M}^X, t)\chi(M_{(X)}, s).
\]
   
Next consider the sum in \eqref{A=B2}. As $X\in M$,  from Theorem \ref{isomo_2}, we obtain
\[(cM)_X= M_X \quad{\rm and}\quad(cM)^X= c(M^X).\] 
Similarly, we have
\[
\chi\big((cM)_X, s\big)=\chi(M_X, s)\quad{\rm and}\quad\chi\big((cM)^X, t\big)=(t-1)\chi(M^X, t).
\]
Since \({\rm rk}_{cM}(X)={\rm rk}_M(X)\) for \(X\in M\), substituting these into \eqref{A=B2} yields
\[
\sum_{X\in M}s^{r+1-{\rm rk}_{cM}(X)}\chi\big((cM)^X, t\big)\chi\big((cM)_X, s\big)
= s(t-1)\sum_{X \in M} s^{r-{\rm rk}_M(X)}\chi(M^X, t)\chi(M_X, s).
\]
   
Now, replacing the summations in \eqref{A=B1} and \eqref{A=B2} with their equivalent forms, and combining Proposition \ref{convolution over semimatroid}, we derive 
\[
(st-1)\chi(M,st)=(s-1)\sum_{X \in \underline{M}}s^{r-{\rm rk}_{\underline{M}}(X)}\chi(\underline{M}^X, t)\chi(M_{(X)}, s)+ s(t-1)\chi(M,st).
\]
This further yields
\[
(s-1)\chi(M,st)=(s-1)\sum_{X \in \underline{M}}s^{r-{\rm rk}_{\underline{M}}(X)}\chi(\underline{M}^X, t)\chi(M_{(X)}, s).
\]
As both sides are bivariate polynomials in $t$ and $s$, we conclude that 
\[
\chi(M,st)=\sum_{X \in \underline{M}}s^{r-{\rm rk}_{\underline{M}}(X)}\chi(\underline{M}^X, t)\chi(M_{(X)},s).
\]
This completes the proof.
\end{proof}
\section{Two combinatorial interpretations}\label{Sec4}
In this section, we focus on finding combinatorial interpretations for the convolution formula in Corollary \ref{cor:arrangement} and Wang's convolution formula \eqref{HA-Convolution}:
\[
\chi(\mathcal{A},st) = \sum_{V \in L(\underline{\mathcal{A}})}\chi(\underline{\mathcal{A}}^V,t)\,\chi(\mathcal{A}_{(V)},s),
\]
and
\[
\chi(\mathcal{A},st) = \sum_{V \in L(\mathcal{A})}\chi(\mathcal{A}^V,t)\,\chi(\mathcal{A}_V,s).
\]
\subsection{A combinatorial interpretation of the first formula}\label{Sec4-1}
This subsection aims to provide a combinatorial interpretation of the first convolution formula using the finite field method. To obtain our desired results, we need to introduce a method based on finite fields for computing the characteristic polynomial of an integral arrangement. An {\em integral hyperplane} $H$ is a hyperplane in $\mathbb{R}^n$ defined by the linear equation 
\[
H: a_1x_1+a_2x_2+\cdots+a_nx_n=b,
\]
where all $a_i$ and $b$ are integers. We abbreviate this defining equation of $H$ as $H:\bm a\cdot\bm x=b$. An {\em integral arrangement} $\mathcal{A}$ is a hyperplane arrangement in $\mathbb{R}^n$ consisting of integral hyperplanes. For every integral hyperplane $H\in\mathcal{A}$, reducing its coefficients modulo a prime $p$ yields an affine subspace $H_q$ in $\mathbb{F}_q^n$, which consists of  all $n$-tuples $(x_1,\ldots,x_n)$ satisfying the defining equation of $H$ over $\mathbb{F}_q$ with $q=p^r (r\ge 1)$. Naturally, for a large enough prime $p$, we obtain a hyperplane arrangement $\mathcal{A}_q$ in $\mathbb{F}_q^n$ consisting of the hyperplanes $H_q$ associated with $\mathcal{A}$ such that the intersection posets of both hyperplane arrangements $\mathcal{A}$ and $\mathcal{A}_q$ are isomorphic, see \cite[Proposition 5.13]{stanley2007hyperplane}. More precisely, the bijection \(H\mapsto H_q\) from \(\mathcal{A}\) to \(\mathcal{A}_q\) yields an order-preserving isomorphism of \(L(\mathcal{A})\) and \(L(\mathcal{A}_q)\). Every flat \(V\in L(\mathcal{A})\) maps bijectively to \(V_q=\bigcap_{H\in\mathcal{A}_V}H_q\in L(\mathcal{A}_q)\) satisfying \((\mathcal{A}_q)_{V_q}=\{H_q:H\in\mathcal{A}_V\}\). Originated from Crapo and Rota's work \cite{Crapo-Rota1970}, Athanasiadis in \cite{athanasiadis1996} systematically developed a finite field method to compute the characteristic polynomial of an integral arrangement. This method reduces the computation of the characteristic polynomial to a simple counting problem in a vector space over a finite field, and was also independently discovered by  Bj\"{o}rner and Ekedahl  \cite{Bjorner-Ekedahl1997}. More precisely, they showed that $\chi(\mathcal{A},q)$ exactly counts the number of all points in $\mathbb{F}_q^n$ that do not lie in any of the hyperplanes in $\mathcal{A}_q$. 
\begin{prop}[ \cite{athanasiadis1996,Bjorner-Ekedahl1997}]\label{FFM}
Let $\mathcal{A}$ be an integral arrangement in $\mathbb{R}^n$. For any large enough prime $p$, we have $L(\mathcal{A})\cong L(\mathcal{A}_q)$ with $q=p^r(r\ge 1)$, and
\[
\chi(\mathcal{A},q)=\#M(\mathcal{A}_q),
\]
where $M(\mathcal{A}_q):=\mathbb{F}_q^n-\bigcup_{H_q\in\mathcal{A}_q}H_q$ is called the complement of $\mathcal{A}_q$.
\end{prop}

Based on Proposition \ref{FFM}, we have the following lemma.
\begin{lemma}\label{Cong1}
Let $\mathcal{A}$ be an integral arrangement in $\mathbb{R}^n$. For any large enough prime $p$, $q=p^r(r\ge 1)$, and  $V \in L(\underline{\mathcal{A}})$, we have
\[
L(\underline{\mathcal A})\cong L(\underline{\mathcal A}_q),\quad L(\mathcal A_{(V)})\cong L\big((\mathcal A_{(V)})_q\big)\quad{\rm and}\quad L(\underline{\mathcal A}^V)\cong L\big((\underline{\mathcal A}^V)_q\big).
\]
Moreover, we have
\[
(\mathcal A_{(V)})_q=(\mathcal A_q)_{(V_q)}\quad{\rm and }\quad (\underline{\mathcal A}^V)_q=\underline{\mathcal A}_q^{V_q}.
\]
\end{lemma}

\begin{proof}
Suppose
\[
\mathcal A=\{H_1,\ldots,H_m\}\quad {\rm and}\quad H_i:\bm a_i\cdot\bm x=b_i,
\]
where each \(\bm a_i\in\mathbb Z^n-\{\bm 0\}\) and \(b_i\in\mathbb Z\). Then
\[
\underline{\mathcal{A}}=\{\underline{H_i}:H_i\in\mathcal{A}\}\quad {\rm and}\quad\underline{H_i}:\bm a_i\cdot\bm x=0.
\]
We first prove that $L(\underline{\mathcal A})\cong L(\underline{\mathcal A}_q)$. For any subset \(I\subseteq \{1,\ldots,m\}\), the dimension of $\bigcap_{i\in I}\underline{H_i}$ is determined by the rank of $\{\bm a_i:i\in I\}$. For all sufficiently large primes \(p\), all nonzero minors remain nonzero modulo \(p\), and hence these ranks are preserved under modulo \(p\). Thus, for every \(I\subseteq\{1,\ldots,m\}\), we have 
\[
\dim \bigcap_{i\in I}\underline{H_i}=\dim \bigcap_{i\in I}\underline{H_{i,q}}.
\]
Moreover, for any \(I,J\subseteq\{1,\ldots,m\}\), the equality $\bigcap_{i\in I}\underline{H_i}=\bigcap_{j\in J}\underline{H_j}$ is equivalent to
\[
\operatorname{rank}(\bm a_i:i\in I)=\operatorname{rank}(\bm a_j:j\in J)=\operatorname{rank}(\bm a_k:k\in I\cup J).
\]
This condition is also preserved modulo \(p\). Thus, the map from \(L(\underline{\mathcal A})\) to \(L(\underline{\mathcal A}_q)\) defined by mapping $\bigcap_{i\in I}\underline{H_i}$ to $\bigcap_{i\in I}\underline{H_{i,q}}$, is an order-preserving bijection. Thus, $L(\underline{\mathcal A})\cong L(\underline{\mathcal A}_q)$.

Since $\mathcal A_{(V)}$ and $(\mathcal A_{(V)})_q$ are the subarrangements of $\mathcal{A}$ and $\mathcal{A}_q$ respectively, $L(\mathcal A_{(V)})$ and $L\big((\mathcal A_{(V)})_q\big)$ are the subposets of $L(\mathcal{A})$ and $L(\mathcal{A}_q)$, respectively. Thus, the isomorphism $L(\mathcal{A})\cong L(\mathcal{A}_q)$ from Proposition \ref{FFM} directly implies $L(\mathcal A_{(V)})\cong L\big((\mathcal A_{(V)})_q\big)$. By analogous arguments, the isomorphism $L(\underline{\mathcal A})\cong L(\underline{\mathcal A}_q)$ also yields the isomorphism $L(\underline{\mathcal A}^V)\cong L\big((\underline{\mathcal A}^V)_q\big)$. Moreover, from the isomorphisms $L(\mathcal{A})\cong L(\mathcal{A}_q)$ and $L(\underline{\mathcal A})\cong L(\underline{\mathcal A}_q)$, we can easily see that $(\mathcal A_{(V)})_q=(\mathcal A_q)_{(V_q)}$ and $(\underline{\mathcal A}^V)_q=\underline{\mathcal A}_q^{V_q}$. 
\end{proof}

Based on the foregoing preparations, we proceed to interpret the first formula combinatorially via the finite field method.
\begin{theorem}
Let $\mathcal{A}$ be an integral arrangement in $\mathbb{R}^n$, and let $p$ be a large enough prime. Then
\[
\chi(\mathcal{A},p^2) = \sum_{V \in L(\underline{\mathcal{A}})}\chi(\underline{\mathcal{A}}^V,p)\chi(\mathcal{A}_{(V)},p).
\]
\end{theorem}
\begin{proof}
Taking an element $\lambda\in \mathbb F_{p^2}-\mathbb F_{p}$, since \([\mathbb F_{p^2}:\mathbb F_p]=2\) and \(\lambda\notin\mathbb F_p\), the set \(\{1,\lambda\}\) forms an \(\mathbb F_p\)-basis of \(\mathbb F_{p^2}\). Hence, each \(\bm z\in\mathbb F_{p^2}^n\) can be written uniquely as \(\bm z=\bm x+\lambda\bm y\), where \(\bm x,\bm y\in\mathbb F_p^n\). This allows us to consider the following mapping 
\[
\phi:\bigsqcup_{V\in L(\underline{\mathcal A})}\Big(M\big((\mathcal A_{(V)})_p\big)\times M\big((\underline{\mathcal A}^V)_p\big)\Big)\to M(\mathcal A_{p^2}),\quad(\bm x,\bm y)\mapsto \bm x+\lambda \bm y.
\]
First note that for any \(\bm x,\bm y\in\mathbb F_p^n\) and $H:\bm a\cdot\bm x=b$ in $\mathcal{A}$, we have the following equivalence:
\[
\bm x+\lambda \bm y\in H_{p^2}\iff[\bm a]_p\cdot\bm x+\lambda[\bm a]_p\cdot\bm y=[b]_p\iff[\bm a]_p\cdot\bm x=[b]_p\ \text{ and  }\ [\bm a]_p\cdot\bm y=[0]_p,
\]
where $[b]_p$ and $[\bm a]_p$ denote the $p$-reductions of $b$ and $\bm a$ modulo the prime $p$, respectively.  This yields the following equivalence:
\begin{equation}\label{A=B3}
\bm x+\lambda \bm y\in H_{p^2}\iff \bm x\in H_p\ \text{and}\ \bm y\in \underline H_p.
\end{equation}

We first verify that \(\phi\) is well-defined by contradiction. Suppose there exists \((\bm x,\bm y)\in M\big((\mathcal A_{(V)})_p\big)\times M\big((\underline{\mathcal A}^V)_p\big)\) such that \(\phi(\bm x,\bm y)\notin M(\mathcal A_{p^2})\). Then
\(\phi(\bm x,\bm y)\in H_{p^2}\) for some \(H\in\mathcal A\). It follows from \eqref{A=B3} that \(\bm x\in H_p\) and \(\bm y\in\underline{H}_p\). Together with \(\bm y\in M\big((\underline{\mathcal A}^V)_p\big)\), the condition \(\bm y\in\underline{H}_p\) implies \(V\subseteq\underline{H}\). Consequently, 
\(H_p\in(\mathcal A_{(V)})_p\), which contradicts \(\bm x\in M\big((\mathcal A_{(V)})_p\big)\). Thus \(\phi(\bm x,\bm y)\in M(\mathcal A_{p^2})\), and hence \(\phi\) is well-defined.

Next, we prove the injectivity of \(\phi\) via contradiction. Suppose there exist distinct elements $(\bm x, \bm y)$ and  $(\bm x', \bm y')$ in $\bigsqcup_{V\in L(\underline{\mathcal A})}\Big(M\big((\mathcal A_{(V)})_p\big)\times M\big((\underline{\mathcal A}^V)_p\big)\Big)$ such that \(\phi(\bm x,\bm y)=\phi(\bm x',\bm y')\), i.e.,  \(\bm x+\lambda \bm y=\bm x'+\lambda \bm y'\). Since the decomposition with respect to the \(\{1,\lambda\}\)  \(\mathbb F_p\)-basis is unique, we derive \(\bm x=\bm x'\) and \(\bm y=\bm y'\), a contradiction. Thus, $\phi$ is an injection.

Finally, we show that $\phi$ is surjective. Let \(\bm z\in M(\mathcal A_{p^2})\), and write \(\bm z=\bm x+\lambda\bm y\) with \(\bm x,\bm y\in\mathbb F_p^n\). Consider the minimal flat $V_{\bm y}\in L(\underline{\mathcal A}_p)$ containing $\bm y$, given by
\[
V_{\bm y}:=\bigcap_{\underline H\in\underline{\mathcal A},\,\bm y\in \underline{H}_p}\underline{H}_p\in L(\underline{\mathcal A}_p).
\]
Let \(V\in L(\underline{\mathcal A})\) be the corresponding flat such that \(V_p=V_{\bm y}\) under the isomorphism $L(\underline{\mathcal{A}})\cong L(\underline{\mathcal{A}}_p)$. We claim that \((\bm x,\bm y)\in M\big((\mathcal A_{(V)})_p\big)\times M\big((\underline{\mathcal A}^V)_p\big)\). It follows that \(\bm y\in V_p\). The minimality of $V_p$ implies that $V_p\subseteq\underline H_p$ whenever $\bm y\in\underline H_p$ for $\underline H\in\underline{\mathcal{A}}$. Thus, we have \(\bm y\in M(\underline{\mathcal A}_p^{V_p})=M\big((\underline{\mathcal A}^V)_p\big)\) from Lemma \ref{Cong1}. On the other hand, suppose \(\bm x\notin M\big((\mathcal A_{(V)})_p\big)\). Then \(\bm x\in H_p\) for some \(H\in\mathcal A_{(V)}\). Since $V$  is either contained in $H$ or is parallel to $H$ from $H\in\mathcal{A}_{(V)}$, we have \(V\subseteq\underline H\). Consequently, \(V_p\subseteq\underline{H}_p\), and hence \(\bm y\in\underline{H}_p\). Applying \eqref{A=B3} once again, we obtain $\bm z=\bm x+\lambda\bm y\in H_{p^2}$, which contradicts the condition $\bm z\in M(\mathcal A_{p^2})$. Consequently, we must have $\bm x\in M\big((\mathcal A_{(V)})_p\big)$. Then, we arrive at $(\bm x,\bm y)\in M\big((\mathcal A_{(V)})_p\big)\times M\big((\underline{\mathcal A}^V)_p\big)$ and $\bm z=\phi(\bm x,\bm y)$. Hence, the map $\phi$ is surjective.

Up to now, we have shown that $\phi$ is a bijection. Thus, the cardinalities of  the sets $\bigsqcup_{V\in L(\underline{\mathcal A})}\Big(M\big((\mathcal A_{(V)})_p\big)\times M\big((\underline{\mathcal A}^V)_p\big)\Big)$ and $M(\mathcal A_{p^2})$ coincide, giving the following identity:
\[
|M(\mathcal A_{p^2})|=\sum_{V\in L(\underline{\mathcal A})}\big|M\big((\mathcal A_{(V)})_p\big)\big|\big|M\big((\underline{\mathcal A}^V)_p\big)\big|.
\]
Applying Proposition \ref{FFM} and Lemma \ref{Cong1}, we conclude that  
\[
\chi(\mathcal{A},p^2) = \sum_{V \in L(\underline{\mathcal{A}})}\chi(\underline{\mathcal{A}}^V,p)\chi(\mathcal{A}_{(V)},p).
\]
This completes the proof.
\end{proof}

\subsection{A combinatorial interpretation of the second formula}\label{Sec4-2}
We now pass from the finite-field interpretation over \(\mathbb F_{p^2}\) and \(\mathbb F_p\) to an analogous interpretation over the rings \(\mathbb Z_{pq}\), \(\mathbb Z_p\), and \(\mathbb Z_q\). In this setting, Wang's convolution formula \cite{wang2015} also admits a natural bijective explanation. Let $\mathcal{A}$ be an integral arrangement in $\mathbb{R}^n$. Similarly, for each integral hyperplane $H:\bm a\cdot\bm x=b$ in $\mathcal{A}$ and a positive integer $q$, they naturally give rise to a coset $H_q$ of $\mathbb{Z}_q^n$ defined by
\[
H_q:=\big\{\bm x\in \mathbb{Z}_q^n: \bm a\cdot\bm x\equiv b\pmod q\big\}.
\]
Collecting all such $H_q$ for $H\in\mathcal{A}$ forms the {\em group arrangement} $\mathcal{A}_q:=\{H_q:H\in\mathcal{A}\}$ in $\mathbb{Z}_q^n$. Likewise, the complement of $\mathcal{A}_q$ is given by
\[
M(\mathcal{A}_q):=\mathbb{Z}_q^n-\bigcup_{H_q\in\mathcal{A}_q}H_q.
\]

In 1999, Athanasiadis generalized the earlier work in \cite[Theorem 2.1]{athanasiadis1999extended}: the characteristic polynomial of $\mathcal{A}$ can be evaluated by counting points in the complement of $\mathcal{A}_q$ for large enough positive integers $q$ coprime to a constant that depends only on $\mathcal{A}$.
\begin{prop}[{\cite[Theorem 2.1]{athanasiadis1999extended}}]\label{prop:Zq-method}
Let \(\mathcal A\) be an integral arrangement in \(\mathbb R^n\). Then there exist positive integers $q_\mathcal{A}$ and $\rho_\mathcal{A}$ depending only on \(\mathcal A\), such that, for every positive integer \(q\) relatively prime to $\rho_\mathcal{A}$ with \(q>q_\mathcal{A}\),
\[
L(\mathcal A)\cong L(\mathcal{A}_q)\quad\;{\rm and}\;\quad\chi(\mathcal A,q)=|M(\mathcal A_q)|.
\]
\end{prop}

Furthermore, $L(\mathcal A)\cong L(\mathcal{A}_q)$ corresponds to the natural order-preserving bijection from $L(\mathcal{A})$ to $L(\mathcal{A}_q)$ that sends $V\in L(\mathcal{A})$ to $V_q=\bigcap_{H\in\mathcal{A}_V}H_q\in L(\mathcal{A}_q)$. Note that for any $V\in L(\mathcal{A})$, $L(\mathcal A_V)$ and $L(\mathcal A^V)$ are the subposets of $L(\mathcal{A})$, while $L\big((\mathcal A_V)_q\big)$ and $L\big((\mathcal A^V)_q\big)$ are the corresponding subposets of $L(\mathcal{A}_q)$. Thus, the isomorphism $L(\mathcal A)\cong L(\mathcal{A}_q)$ directly induces the two isomorphisms given in Lemma \ref{Cong2}. 
\begin{lemma}\label{Cong2}
Let $\mathcal{A}$ be an integral arrangement in $\mathbb{R}^n$. If a positive integer $q>q_{\mathcal{A}}$ and $(q,\rho_\mathcal{A})=1$, then for any $V\in L(\mathcal{A})$,
\[
L(\mathcal A_V)\cong L\big((\mathcal A_V)_q\big)\quad{\rm and}\quad L(\mathcal A^V)\cong L\big((\mathcal A^V)_q\big).
\]
Moreover, we have
\[
(\mathcal A_V)_q=(\mathcal A_q)_{V_q}\quad{\rm and }\quad (\mathcal A^V)_q=\mathcal A_q^{V_q}.
\]
\end{lemma}

Based on Proposition  \ref{prop:Zq-method}, below we provide a combinatorial interpretation of Wang's convolution formula.
\begin{theorem}\label{thm:Zpq-convolution}
Let \(\mathcal A\) be an integral arrangement in \(\mathbb R^n\). Suppose positive integers $p>q_\mathcal{A}$ and $q>q_\mathcal{A}$ satisfy $(p,q)=1$, 
$(p,\rho_\mathcal{A})=1$ and $(q,\rho_\mathcal{A})=1$. Then 
\[
\chi(\mathcal A,pq)=\sum_{V\in L(\mathcal A)}\chi(\mathcal A^V,p)\chi(\mathcal A_V,q).
\]
\end{theorem}

\begin{proof}
Since \(p\) and \(q\) are coprime, we have the natural isomorphism $\mathbb{Z}_{pq}\cong\mathbb{Z}_p\times\mathbb{Z}_q$ by the Chinese Remainder Theorem. This automatically gives rise to the following isomorphism:
\[
\theta:\mathbb Z_{pq}^n\xrightarrow{\sim}\mathbb Z_p^n\times\mathbb Z_q^n,\qquad[\bm z]_{pq}\mapsto([\bm z]_p, [\bm z]_q) \quad{\rm for}\quad \bm z\in\mathbb{Z}^n.
\]
Consequently, for any $\bm x\in\mathbb{Z}_p^n,\bm y\in\mathbb{Z}_q^n,\bm z\in\mathbb{Z}_{pq}^n$, if 
$\theta(\bm z)=(\bm x,\bm y)$, then 
\begin{equation}\label{A=B4}
\bm z\in H_{pq}\iff\bm x\in H_p\;\text{ and }\;\bm y\in H_q.
\end{equation}
The mapping $\theta$ naturally induces the following map: 
\[
\psi:\bigsqcup_{V\in L(\mathcal A)}\Big(M\big((\mathcal A^V)_p\big)\times M\big((\mathcal A_V)_q\big)\Big)\to M(\mathcal A_{pq}),\quad(\bm x,\bm y)\mapsto \theta^{-1}(\bm x,\bm y).
\]

First, we prove that \(\psi\) is well-defined by contradiction. Suppose there exist \(V\in L(\mathcal A)\) and \((\bm x,\bm y)\in M\big((\mathcal A^V)_p\big)\times M\big((\mathcal A_V)_q\big)\) such that \(\psi(\bm x,\bm y)\notin M(\mathcal A_{pq})\). Then \(\psi(\bm x,\bm y)\in H_{pq}\) for some \(H\in\mathcal A\). According to \eqref{A=B4}, we have \(\bm x\in H_p\) and \(\bm y\in H_q\). Since \(\bm x\in M\big((\mathcal A^V)_p\big)=M(\mathcal A_p^{V_p})\subseteq V_p\) from Lemma \ref{Cong2}, the condition \(\bm x\in H_p\) implies \(V\subseteq H\). Thus, \(H_q\in(\mathcal A_V)_q\), and hence \(\bm y\notin M\big((\mathcal A_V)_q\big)\), a contradiction. So \(\psi\) is well-defined.
	
It is obvious that the injectivity of \(\psi\) directly follows from the injectivity of \(\theta\). Next, we show that $\psi$ is surjective. From the definition of $\psi$, it remains to prove that for any $\bm z\in M(\mathcal{A}_{pq})$, $\psi^{-1}(\bm z)=\theta(\bm z)=(\bm x,\bm y)\in M\big((\mathcal A^V)_p\big)\times M\big((\mathcal A_V)_q\big)$ for some \(V\in L(\mathcal A)\), i.e., $\psi^{-1}$ is well-defined. Consider the minimal flat $V_{\bm x}\in L(\mathcal A_p)$ containing $\bm x$, given by
\[
V_{\bm x}:=\bigcap_{H\in\mathcal A,\, \bm x\in H_p}H_p \in L(\mathcal A_p).
\]
From Proposition \ref{prop:Zq-method}, we have $L(\mathcal{A})\cong L(\mathcal{A}_{p})$. Accordingly, let \(V\in L(\mathcal A)\)  be the corresponding flat with \(V_p=V_{\bm x}\) under the lattice isomorphism $L(\mathcal{A})\cong L(\mathcal{A}_{p})$. It follows that  \(\bm x\in V_p\). The minimality of $V_p$ implies that $V_p\subseteq H_p$ whenever $\bm x\in H_p$ for any $H\in\mathcal{A}$. Consequently, we have $\bm x\in M(\mathcal A_p^{V_p})=M\big((\mathcal A^V)_p\big)$ via Lemma \ref{Cong2}. It suffices to prove that \(\bm y\in M\big((\mathcal A_V)_q\big)\). Otherwise, there exists $H\in\mathcal{A}$ such that $V\subseteq H$ and \(\bm y\in H_q\). Combining \(\bm x\in V_p\subseteq H_p\), we deduce
\(\bm z\in H_{pq}\) by \eqref{A=B4}, contradicting \(\bm z\in M(\mathcal A_{pq})\). Thus, $\bm y\in M\big((\mathcal A_V)_q\big)$, and hence \(\psi\) is indeed surjective.

We have now shown that $\psi$ is a bijection. Thus, the cardinalities of  the sets $M(\mathcal A_{pq})$ and $\bigsqcup_{V\in L(\mathcal A)}\Big(M\big((\mathcal A^V)_p\big)\times M\big((\mathcal A_V)_q\big)\Big)$ coincide, i.e.,
\[
|M(\mathcal A_{pq})|=\sum_{V\in L(\mathcal A)}\big|M\big((\mathcal A^V)_p\big)\big|\big|M\big((\mathcal A_V)_q\big)\big|.
\]
Since $p,q>q_\mathcal{A}$, $(p,q)=1$, $(p,\rho_\mathcal{A})=1$ and $(q,\rho_\mathcal{A})=1$, we have $pq>q_\mathcal{A}$ and $(pq,\rho_\mathcal{A})=1$. Combining with Proposition \ref{prop:Zq-method} and Lemma \ref{Cong2}, we deduce that  
\[
\chi(\mathcal A,pq)=\sum_{V\in L(\mathcal A)}\chi(\mathcal A^V,p)\chi(\mathcal A_V,q).
\]
This completes the proof.
\end{proof}

It is natural to conjecture that the convolution formula established in Corollary \ref{cor:arrangement} should also admit a corresponding interpretation over $\mathbb{Z}_{pq}$, $\mathbb{Z}_p$, and $\mathbb{Z}_q$. To date, however, we have not been able to find a natural bijection that proves this. We therefore record this as an open problem, in the hope that it may be resolved by other researchers in the future.

\section*{Acknowledgements}
This work is supported by the National Natural Science Foundation of China (Grant No. 12571350) and the Guangdong Basic and Applied Basic Research Foundation (Grant No. 2026A1515012237, Grant No. 2025A1515010457). 

\providecommand{\bysame}{\leavevmode\hbox to3em{\hrulefill}\thinspace}
  \providecommand{\MR}{\relax\ifhmode\unskip\space\fi MR }
  \providecommand{\MRhref}[2]{%
    \href{http://www.ams.org/mathscinet-getitem?mr=#1}{#2}}
  \providecommand{\href}[2]{#2}
 
  \end{document}